\documentclass[12pt,a4paper]{article}	

\usepackage[english]{babel}
\usepackage{amsmath}
\usepackage{amssymb}
\usepackage{mathdots}
\usepackage{mathabx}
\usepackage{tikz}
\usepackage{tikz-cd}
\usepackage{enumitem}
\usepackage[whole]{bxcjkjatype}
\usepackage{csquotes}
\usepackage{changepage}
\usepackage{pdflscape}
\usepackage{geometry}
\usepackage{color}
\definecolor{linkcolor}{rgb}{0,0,0.6}
\usepackage[colorlinks=true,
pdfstartview=FitV,
linkcolor= linkcolor,
citecolor= linkcolor,
urlcolor= linkcolor,
hyperindex=true,
hyperfigures=false]
{hyperref}
\usepackage{scalerel}
\usepackage{amsthm}
\title{Cohomology of the Bruhat-Tits strata in the supersingular locus of the $\mathrm{GU}(1,n-1)$ Shimura variety at a ramified prime}
\author{J.Muller}
\date{}

\usepackage[
backend=bibtex,
]{biblatex}
\usepackage{ytableau}

\begin{document}

%Theorem environments
\newtheorem{theo}{Theorem}
\newtheorem{prop}[theo]{Proposition}
\newtheorem{lem}[theo]{Lemma}
\newtheorem{corol}[theo]{Corollary}
\newtheorem{conj}[theo]{Conjecture}
\newtheorem*{theo*}{Theorem}

\theoremstyle{remark}
\newtheorem{rk}[theo]{Remark}
\newtheorem{rks}[theo]{Remarks}
\newtheorem{ex}[theo]{Example}

\theoremstyle{definition}
\newtheorem{defi}[theo]{Definition}
\newtheorem*{notation}{Notation}
\newtheorem*{notations}{Notations}

%Title and abstract

\maketitle

\begin{center}

\parbox{15cm}{\small
\textbf{Abstract} : \it The supersingular locus of the $\mathrm{GU}(1,n-1)$ Shimura variety at a ramified prime $p$ is stratified by Coxeter varieties attached to finite symplectic groups. In this paper, we compute the $\ell$-adic cohomology of the Zariski closure of any such stratum. These are referred to as closed Bruhat-Tits strata. We prove that the cohomology groups of odd degree vanish, and those of even degree are explicitely determined as representations of the symplectic group with Frobenius action. Our approach is based on the spectral sequence induced by the stratification by Coxeter varieties, whose cohomology have been computed by Lusztig. We describe explicitely the terms at infinity in the sequence. We point out that closed Bruhat-Tits strata have isolated singularities when the dimension is greater than 1. Our analysis requires discussing the smoothness of the blow-up at the singular points, as well as comparing the ordinary $\ell$-adic cohomology with intersection cohomology. A by-product of our computations is that these two cohomologies actually coincide.}

\vspace{0.5cm}
\end{center}

\tableofcontents

\vspace{1.5cm}

\textbf{\textsc{Introduction:}} Shimura varieties are objects of central interest in number theory and arithmetic geometry. When the Shimura variety admits an integral model over the completion at some prime of the underlying number field, one may be interested in the geometry of its special fiber. In particular, the supersingular locus of the special fiber has been extensively studied in recent years. In many situations which are precisely listed in \cite{GHNcoxeter} and \cite{GHNhodge}, the supersingular locus admits a so-called Bruhat-Tits stratification, whose strata are isomorphic to classical Deligne-Lusztig varieties for certain finite groups of Lie type. Cohomology plays an important role both in the world of Shimura varieties and in Deligne-Lusztig theory. Thus, exploiting the geometry of the Bruhat-Tits stratification in order to connect both cohomology theories sounds like a promising idea. In \cite{mullerBT} and \cite{mullerUnramified}, we investigated the case of the $\mathrm{GU}(1,n-1)$ PEL Shimura variety over an inert prime $p > 2$, whose supersingular locus was described in \cite{vw1} and \cite{vw2}. More precisely, we explicitely determined the cohomology of the closed Bruhat-Tits strata as representations of finite unitary groups. We used this result to prove the non-admissibility of the cohomology of the associated Rapoport-Zink space, and to determine the cohomology of the supersingular locus for low $n$ in terms of automorphic representations. In this paper, we focus on the case of a ramified prime $p>2$ as studied in \cite{RTW}. Our goal is to replicate the same approach as in the inert case, and find out how to deal with the new technical difficulties caused by the non-smoothness of the integral model. The exposition is divided into two papers, and the present paper is the first of the series. It is devoted to the computation of the cohomology of a given closed Bruhat-Tits stratum using Deligne-Lusztig theory. Let us explain the results in more details.\\
Let $q$ be a power of an odd prime number $p$. Let $V$ be a symplectic space over $\mathbb F_q$ of dimension $2\theta$. For any field extension $k/\mathbb F_q$, let $\tau = \mathrm{id}\otimes \sigma$ denote the semi-linear automorphism of $V_k := V\otimes k$, where $\sigma:x\mapsto x^q$. Let $L(V)$ denote the Lagrangian Grassmanian variety of $V$. We consider the closed subvariety $S_{\theta} \subset L(V)$ whose $k$-points are given by 
$$S_{\theta}(k) = \{U \subset V_k \,|\, U = U^{\perp} \text{ and } \dim(U\cap \tau(U)) \geq \theta-1\}.$$
It turns out that the closed Bruhat-Tits strata mentioned in the previous paragraph are actually isomorphic to $S_{\theta}$ for some $\theta \geq 0$. The variety $S_{\theta}$ is projective, irreducible, normal and of dimension $\theta$. It has isolated singularities when $\theta \geq 2$. When $\theta = 1$, we have $S_1 \simeq \mathbb P^1$. Moreover it is equipped with a natural action of the finite symplectic group $\mathrm{Sp}(V)$. Up to fixing a basis of $V$, we may identify $\mathrm{Sp}(V)$ with the usual group $\mathrm{Sp}(2\theta,\mathbb F_q)$ of symplectic matrices. Unipotent representations of $\mathrm{Sp}(2\theta,\mathbb F_q)$ are naturally classified by Lusztig's notion of symbols, whose definition is recalled in Section \ref{Section2}. Given a symbol $S$, the associated unipotent representation is denoted $\rho_S$. The main theorem is the following.

\begin{theo*}
\begin{enumerate}[label=\upshape (\arabic*), topsep = 0pt]
		\item All the cohomology groups of $S_{\theta}$ of odd degree vanish.
		\item For $0 \leq i \leq \theta$, we have an $\mathrm{Sp}(2\theta,\mathbb F_q)$-equivariant isomorphism
		$$\mathrm H^{2i}(S_{\theta},\overline{\mathbb Q_{\ell}}) \simeq \bigoplus_{s=0}^{\min(i,\theta-i)} \rho_{\scaleto{
		\begin{pmatrix}
		s & \theta+1-s \\
		0 & {} 
		\end{pmatrix}}{25pt}} \oplus \bigoplus_{s = 0}^{\min(i-1,\theta-1-i)} \rho_{\scaleto{
		\begin{pmatrix}
		0 & s + 1 & \theta-s \\
		{} & {} & {} 
		\end{pmatrix}}{25pt}}.$$
		The Frobenius acts like multiplication by $q^i$ on the first summand, and multiplication by $-q^i$ on the second summand.
	\end{enumerate}
\end{theo*}

Let us explain the main steps of the proof. The variety $S_{\theta}$ admits a stratification 
$$S_{\theta} = \bigsqcup_{\theta'=0}^{\theta} X_{I_{\theta'}}(w_{\theta'}),$$
where the $X_{I_{\theta'}}(w_{\theta'})$ are certain Deligne-Lusztig varieties which are \enquote{parabolically induced} from the Coxeter variety $X^{\theta'}$ of the finite group $\mathrm{Sp}(2\theta',\mathbb F_q)$, see Section \ref{Section1} for the precise definitions. There is an induced spectral sequence 
$$E^{\theta',i}_1 = \mathrm H^{\theta'+i}_c(X_{I_{\theta'}}(w_{\theta'})) \implies \mathrm H^{\theta'+i}(S_{\theta}).$$
See Figure \ref{Figure1} for a drawing of $E_1$. The term $E^{\theta',i}_1$ is the parabolic induction of the degree $i$ cohomology group of the Coxeter variety $X^{\theta'}$. The cohomology of such Coxeter varieties has been computed in \cite{cox}, and the parabolic inductions can be computed explicitely via the comparison theorem of \cite{howlett}. In particular, we can determine $E^{\theta',i}_1$ explicitely, see Lemma \ref{FirstTerms}. The Frobenius acts semi-simply on $E_1^{\theta',i}$ with at most $2$ eigenvalues. These eigenvalues are equal to $q^i$ and $-q^{i+1}$, the latter only occuring if $0 \leq i\leq \theta'-2$. Since terms on different rows do not carry any common eigenvalue, the spectral sequence degenerates in $E_2$ and the resulting filtration on the abutment splits. Thus, we are reduced to computing the terms $E_2^{\theta',i}$ explicitely.\\
The most effective way to determine most of the terms $E_2^{\theta',i}$ is to find restrictions on the eigenvalues of the Frobenius on the abutment of the spectral sequence. To do so, we seek a good resolution of the singularities of $S_{\theta}$ when $\theta > 1$. Such a resolution is afforded by the blow-up at singular points, as we prove in Section \ref{Section5}. In fact, we exhibit a certain affine open neighborhood of any given singular point, and observe that it is a finite étale cover of the symmetric determinantal variety of rank $\leq 1$. Incidentally, desingularizations of such determinantal varieties have been studied in \cite{blowup} by means of successive blow-ups. In particular, in our case a single blow-up is required to resolve the singularities. As a consequence, we prove that the Frobenius action on $\mathrm H^k(S_{\theta},\overline{\mathbb Q_{\ell}})$ is pure of weight $2\lfloor\frac{k}{2}\rfloor$. In particular, all terms $E_2^{\theta',i}$ which do not carry any eigenvalue of compatible weight must vanish.\\
In order to determine the remaining terms $E_2^{\theta',i}$, we introduce the intersection cohomology of $S_{\theta}$. Since $S_{\theta}$ has only isolated singularities (when $\theta > 1$), it is well-known that intersection cohomology and $\ell$-adic cohomology agree above the middle degree. In particular, the cohomology groups $\mathrm H^k(S_{\theta})$ for $k > \theta$ with $k$ odd vanish, since all weights of the Frobenius are both odd (by intersection cohomology) and even (by the spectral sequence) at the same time. The hypercohomology spectral sequence associated to the intersection complex, as represented in Figure \ref{Figure2}, allows us to remove the restriction $k > \theta$, thus proving the first part of the main theorem. The second part follows easily given the shape of the spectral sequence represented in Figure \ref{Figure1}. As a by-product, we find out that the intersection complex of $S_{\theta}$ has vanishing cohomology in higher degrees. In particular, the intersection cohomology and the $\ell$-adic cohomology of $S_{\theta}$ actually agree in all degrees.\\

\textbf{\textsc{Notations:}} In this paper, $p$ will always denote an odd prime number and $q$ will be a power of $p$. If $M$ is a matrix with coefficients in a field of characteristic $p$, then $M^{(q)}$ denotes the matrix $M$ with entries raised to the power $q$. The trivial representation of a given group will be denoted $\mathbf 1$. Given a reductive group $G$ with Levi complement $L$, the associated Harish-Chandra induction and restriction functors are denoted $\mathrm R_L^G$ and ${}^*\mathrm R_L^G$ respectively. \\

\textbf{\textsc{Acknowledgement:}} This paper is part of a PhD thesis under the supervision of Pascal Boyer and Naoki Imai. I am grateful for their wise guidance throughout the research.

\section{The closed Deligne-Lusztig variety $S_{\theta}$ for $\mathrm{Sp}(2\theta,\mathbb F_q)$}
\label{Section1}

Let $\mathbf G$ be a connected reductive group over $\mathbb F$, together with a split $\mathbb F_q$-structure given by a geometric Frobenius morphism $F$. For $\mathbf H$ any $F$-stable subgroup of $\mathbf G$, we write $H := \mathbf H^F$ for its group of $\mathbb F_q$-rational points. Let $(\mathbf T,\mathbf B)$ be a pair consisting of a maximal $F$-stable torus $\mathbf T$ contained in an $F$-stable Borel subgroup $\mathbf B$. Let $(\mathbf W,\mathbf S)$ be the associated Coxeter system, where $\mathbf W = \mathrm{N}_{\mathbf G}(\mathbf T)/\mathbf T$. Since the $\mathbb F_q$-structure on $\mathbf G$ is split, the Frobenius $F$ acts trivially on $\mathbf W$. For $I\subset \mathbf S$, let $\mathbf P_I, \mathbf U_I, \mathbf L_I$ be respectively the standard parabolic subgroup of type $I$, its unipotent radical and its unique Levi complement containing $\mathbf T$. Let $\mathbf W_I$ be the subgroup of $\mathbf W$ generated by $I$. For $\mathbf P$ any parabolic subgroup of $\mathbf G$, the associated \textbf{generalized parabolic Deligne-Lusztig variety} is
$$X_{\mathbf P} := \{g\mathbf P\in \mathbf G/\mathbf P\,|\,g^{-1}F(g)\in \mathbf P F(\mathbf P)\}.$$
We say that the variety is \textbf{classical} (as opposed to generalized) when in addition the parabolic subgroup $\mathbf P$ contains an $F$-stable Levi complement. Note that $\mathbf P$ itself need not be $F$-stable. We may give an equivalent definition using the Coxeter system $(\mathbf W,\mathbf S)$. For $I\subset \mathbf S$, let ${}^I\mathbf W^{I}$ be the set of elements $w\in \mathbf W$ which are $I$-reduced-$I$. For $w\in\, ^I\mathbf W^{I}$, the associated generalized parabolic Deligne-Lusztig variety is
$$X_I(w)  := \{g\mathbf P_I\in \mathbf G/\mathbf P_I \,|\, g^{-1}F(g)\in \mathbf P_Iw\mathbf P_I\}.$$
The variety $X_I(w)$ is classical when $w^{-1}Iw = I$, and it is defined over $\mathbb F_q$. The dimension is given by $\dim X_I(w) = l(w)+ \dim \mathbf G/\mathbf P_{I\cap wIw^{-1}} - \dim \mathbf G/\mathbf P_I$ where $l(w)$ denotes the length of $w$ with respect to $\mathbf S$.\\
%A morphism $f:\mathbf G \to \mathbf G'$ between two connected reductive groups over $\mathbb F$, both equipped with $\mathbb F_q$-structures induced by Frobenius morphisms $F$ and $F'$ respectively, is said to be an $\mathbb F_q$-isotypy if it is defined over $\mathbb F_q$, if its kernel is contained in the center of $\mathbf G$ and if its image contains the derived subgroup of $\mathbf G'$. Given such an $\mathbb F_q$-isotypy $f$ we have $\mathbf G' = f(\mathbf G)\mathrm{Z}(\mathbf G')^{0}$, where $\mathrm{Z}(\mathbf G')^{0}$ is the identity component of the center of $\mathbf G'$. Thus intersecting with $f(\mathbf G)$ defines a bijection between parabolic subgroups of $\mathbf G'$ and those of $f(\mathbf G)$. Let $\mathbf P$ be a parabolic subgroup of $\mathbf G$ and let $\mathbf P' = f(\mathbf P)\mathrm{Z}(\mathbf G')^{0}$ be the corresponding parabolic subgroup of $\mathbf G'$. Then the map $g\mathbf P \mapsto f(g\mathbf P)$ induces an isomorphism $f:X_{\mathbf P} \xrightarrow{\sim} X_{\mathbf P'}$ which is compatible with the actions of $G$ and $G'$ via $f$. Therefore $\mathbf G$ and $\mathbf G'$ generate the same Deligne-Lusztig varieties.\\
Let $\theta\geq 0$ and let $V$ be a $2\theta$-dimensional $\mathbb F_{q}$-vector space equipped with a non-degenerate symplectic form $(\cdot,\cdot):V\times V \to \mathbb F_{q}$. Fix a basis $(e_1,\ldots,e_{2\theta})$ in which $(\cdot,\cdot)$ is described by the matrix 
$$\Omega := \begin{pmatrix}
0 & A_{\theta} \\
-A_{\theta} & 0
\end{pmatrix},$$
where $A_{\theta}$ denotes the matrix having $1$ on the anti-diagonal and $0$ everywhere else. If $k$ is a field extension of $\mathbb F_q$, let $V_k := V \otimes_{\mathbb F_q} k$ denote the scalar extension to $k$ equipped with its induced $k$-symplectic form $(\cdot,\cdot)$. Let $\tau:V_k \xrightarrow{\sim} V_k$ denote the map $\mathrm{id}\otimes \sigma$, where $\sigma(x) := x^q$. If $U \subset V_k$, let $U^{\perp}$ denote its orthogonal. We consider the finite symplectic group $\mathrm{Sp}(V,(\cdot,\cdot)) \simeq \mathrm{Sp}(2\theta,\mathbb F_q)$, where the RHS is defined with respect to $\Omega$. It can be identified with $G = \mathbf G^F$ where $\mathbf G$ is the symplectic group $\mathrm{Sp}(V_{\mathbb F},(\cdot,\cdot)) \simeq \mathrm{Sp}(2\theta,\mathbb F)$ and $F(M) := M^{(q)}$. Let $\mathbf T \subset \mathbf G$ be the maximal torus of diagonal symplectic matrices and let $\mathbf B \subset \mathbf G$ be the Borel subgroup of upper-triangular symplectic matrices. The Weyl system of $(\mathbf T,\mathbf B)$ is identified with $(W_{\theta},\mathbf S)$ where $W_{\theta}$ is the finite Coxeter group of type $B_{\theta}$ and $\mathbf S = \{s_1,\ldots ,s_{\theta}\}$ is the set of simple reflexions. They satisfy the following relations 
\begin{align*}
(s_{\theta}s_{\theta-1})^4 & = 1, & (s_is_{i-1})^3 & = 1, & & \forall \; 2\leq i \leq \theta-1, \\
(s_is_j)^2 & = 1, & & & & \forall\; |i-j| \geq 2.
\end{align*} 
Concretely, the simple reflexion $s_i$ acts on $V$ by exchanging $e_i$ and $e_{i+1}$ as well as $e_{2\theta-i}$ and $e_{2\theta-i+1}$ for $1\leq i \leq \theta-1$, whereas $s_{\theta}$ exchanges $e_{\theta}$ and $e_{\theta+1}$. We define 
$$I:=\{s_1,\ldots ,s_{\theta-1}\} = \mathbf S\setminus \{s_{\theta}\}.$$
We consider the generalized Deligne-Lusztig variety $X_{I}(s_{\theta})$. Since $s_{\theta}s_{\theta-1}s_{\theta} \not \in I$, it is not a classical Deligne-Lusztig variety. Let $S_{\theta} := \overline{X_{I}(s_{\theta})}$ be its closure in $\mathbf G/\mathbf P_I$. This variety has been introduced in \cite{RTW}, as it is isomorphic to the closed Bruhat-Tits strata of type $\theta$ in the supersingular locus of the $\mathrm{GU}(1,n-1)$ Shimura variety over a ramified prime. We recall some geometric facts on $S_{\theta}$ which are proved in \cite{RTW} Section 5.

\begin{prop}\label{RationalPointsStheta}
The variety $S_{\theta}$ is normal and projective. The variety $S_1$ is isomorphic to the projective line $\mathbb P^1$, and for $\theta \geq 2$ the variety $S_{\theta}$ has isolated singularities. If $k$ is a field extension of $\mathbb F_q$, we have
$$S_{\theta}(k) \simeq \{U \subset V_k \,|\, U^{\perp} = U \text{ and }U\cap \tau(U) \overset{\leq 1}{\subset} U\},$$
where $\overset{\leq 1}{\subset}$ denotes an inclusion of subspaces with index at most $1$. There is a decomposition 
$$S_{\theta} = X_{I}(\mathrm{id}) \sqcup X_{I}(s_{\theta}),$$
where $X_{I}(\mathrm{id})$ is closed and of dimension $0$, and $X_{I}(s_{\theta})$ is open, dense of dimension $\theta$. They correspond respectively to $k$-points $U$ having $U = \tau(U)$ or $U\cap \tau(U) \subsetneq U$. If $\theta\geq 2$ then $S_{\theta}$ is singular at the points of $X_I(\mathrm{id})$.
\end{prop}

For $0 \leq \theta' \leq \theta$, define 
$$I_{\theta'} := \{s_1,\ldots ,s_{\theta-\theta'-1}\}$$
and $w_{\theta'} := s_{\theta+1-\theta'}\ldots s_{\theta}$. In particular $I_0 = I$, $I_{\theta-1} = I_{\theta} = \emptyset$, $w_0 = \mathrm{id}$ and $w_1 = s_{\theta}$. 

\begin{prop}\label{Stratification}
There is a stratification into locally closed subvarieties
$$S_{\theta} = \bigsqcup_{\theta'=0}^{\theta} X_{I_{\theta'}}(w_{\theta'}).$$
The stratum $X_{I_{\theta'}}(w_{\theta'})$ corresponds to points $U$ such that $\dim(U\cap \tau(U) \cap \ldots \cap \tau^{\theta'+1}(U)) = \theta - \theta'$. The closure in $S_{\theta}$ of a stratum $X_{I_{\theta'}}(w_{\theta'})$ is the union of all the strata $X_{I_{t}}(w_{t})$ for $t\leq \theta'$. The stratum $X_{I_{\theta'}}(w_{\theta'})$ is of dimension $\theta'$, and $X_{I_{\theta}}(w_{\theta})$ is open, dense and irreducible.
\end{prop}

It follows in particular that $S_{\theta}$ is irreducible as well. It turns out that the strata $X_{I_{\theta'}}(w_{\theta'})$ are related to Coxeter varieties for symplectic groups of smaller sizes. For $0 \leq \theta' \leq \theta$, define
$$K_{\theta'} := \{s_1,\ldots s_{\theta-\theta'-1}, s_{\theta-\theta'+1}, \ldots , s_\theta\} = \mathbf S \setminus \{s_{\theta-\theta'}\}.$$
Note that $K_0 = I_0 = I$ and $K_{\theta} = \mathbf S$. We have $I_{\theta'} \subset K_{\theta'}$ with equality if and only if $\theta'=0$.

\begin{prop}\label{TransitivityIdentity}
There is an $\mathrm{Sp}(2\theta,\mathbb F_p)$-equivariant isomorphism
$$X_{I_{\theta'}}(w_{\theta'}) \simeq \mathrm{Sp}(2\theta,\mathbb F_q)/U_{K_{\theta'}} \times_{L_{K_{\theta'}}} X_{I_{\theta'}}^{\mathbf L_{K_{\theta'}}}(w_{\theta'}),$$
where $X_{I_{\theta'}}^{\mathbf L_{K_{\theta'}}}(w_{\theta'})$ is a Deligne-Lusztig variety for $\mathbf L_{K_{\theta'}}$. The zero-dimensional variety $\mathrm{Sp}(2\theta,\mathbb F_q)/U_{K_{\theta'}}$ has a left action of $\mathrm{Sp}(2\theta,\mathbb F_q)$ and a right action of $L_{K_{\theta'}}$.
\end{prop}

\begin{proof}
It is a special case of the geometric identity used to prove transitivity of the Deligne-Lusztig induction functor. We refer to \cite{dm} Proposition 7.19 or \cite{dmbook} Proposition 9.1.8. 
\end{proof}

The Levi complement $\mathbf L_{K_{\theta'}}$ is isomorphic to $\mathrm{GL}(\theta-\theta') \times \mathrm{Sp}(2\theta')$, and its Weyl group is isomorphic to $\mathfrak S_{\theta-\theta'}\times W_{\theta'}$. Via this decomposition, the permutation $w_{\theta'}$ corresponds to $\mathrm{id}\times w_{\theta'}$. The Deligne-Lusztig variety $X_{I_{\theta'}}^{\mathbf L_{K_{\theta'}}}(w_{\theta'})$ decomposes as a product 
$$X_{I_{\theta'}}^{\mathbf L_{K_{\theta'}}}(w_{\theta'}) = X_{\mathbf I_{\theta'}}^{\mathrm{GL(\theta-\theta')}}(\mathrm{id}) \times X_{\emptyset}^{\mathrm{Sp}(2\theta')}(w_{\theta'}).$$
The variety $X_{\mathbf I_{\theta'}}^{\mathrm{GL(\theta-\theta')}}(\mathrm{id})$ is just a single point, but $X_{\emptyset}^{\mathrm{Sp}(2\theta')}(w_{\theta'})$ is the Coxeter variety for the symplectic group of size $2\theta'$. Indeed, $w_{\theta'}$ is a Coxeter element, ie. the product of all the simple reflexions of the Weyl group of $\mathrm{Sp}(2\theta')$.

\section{Unipotent representations of the finite symplectic group}
\label{Section2}

Recall that a (complex) irreducible representation of a finite group of Lie type $G = \mathbf G^{F}$ is said to be \textbf{unipotent}, if it occurs in the Deligne-Lusztig induction of the trivial representation of some maximal rational torus. Equivalently, it is unipotent if it occurs in the cohomology (with coefficient in $\overline{\mathbb Q_{\ell}}$ and $\ell \not = p$) of some Deligne-Lusztig variety of the form $X_{\mathbf B}$, with $\mathbf B$ a Borel subgroup of $\mathbf G$ containing a maximal rational torus.
%Let $\mathbf G, \mathbf G'$ and let $f:\mathbf G \to \mathbf G'$ be an $\mathbb F_q$-isotypy as in \ref{SameDLVarieties}. If $\mathbf B$ is such a Borel in $\mathbf G$, then $\mathbf B' := f(\mathbf B)\mathrm{Z}(\mathbf G')^{0}$ is such a Borel in $\mathbf B'$, and $f$ induces an isomorphism $X_{\mathbf B} \xrightarrow{\sim} X_{\mathbf B'}$ compatible with the actions. As a consequence, the map 
%$$\rho \mapsto f\circ \rho$$
%defines a bijection between the sets of equivalence classes of unipotent representations of $G'$ and of $G$. We will use this observation later in the case $\mathbf G = \mathrm{Sp}(2\theta)$ and $\mathbf G' = \mathrm{GSp}(2\theta)$, the symplectic group and the group of symplectic similitudes, the morphism $f$ being the inclusion. 
In this section, we recall the classification of the unipotent representations of the finite symplectic groups. The underlying combinatorics is described by Lusztig's notion of symbols. Our main reference is \cite{geck} Section 4.4.

\begin{defi}
Let $\theta \geq 1$ and let $d$ be an odd positive integer. The set of \textbf{symbols of rank $\theta$ and defect $d$} is 
$$\hspace{-30pt}\mathcal Y^1_{d,\theta} := \left\{ S = (X,Y) \,\bigg |\, \begin{array}{l}
X = (x_1,\ldots ,x_{r+d})\\
Y = (y_1,\ldots ,y_r)
\end{array},
x_i,y_j \in \mathbb Z_{\geq 0}, 
\begin{array}{l}
x_{i+1}-x_i \geq 1,\\
y_{j+1}-y_j \geq 1,
\end{array}
\mathrm{rk}(S) = \theta\right\} \bigg / (\text{shift}),$$
where the shift operation is defined by $\text{shift}(X,Y) := (\{0\}\sqcup(X+1),\{0\}\sqcup(Y+1))$, and where the rank of $S$ is given by
$$\mathrm{rk}(S) := \sum_{s\in S} s - \left\lfloor\frac{(\#S-1)^2}{4}\right\rfloor.$$
\end{defi}

Note that the formula defining the rank is invariant under the shift operation, therefore it is well defined. By \cite{classical}, we have $\mathrm{rk}(S) \geq \left\lfloor\frac{d^2}{4}\right\rfloor$ so in particular $\mathcal Y^1_{d,\theta}$ is empty for $d$ big enough. We write $\mathcal Y^1_{\theta}$ for the union of the $\mathcal Y^1_{d,\theta}$ with $d$ odd, this is a finite set.

\begin{ex}\label{ExempleSymbols}
In general, a symbol $S = (X,Y)$ will be written 
$$S = \setlength\arraycolsep{2pt}
\begin{pmatrix}
x_1 & \ldots & x_{r} & \ldots & x_{r+d}\\
y_1 & \ldots & y_{r} & &
\end{pmatrix}.$$
We refer to $X$ and $Y$ as the first and second rows of $S$. The $6$ elements of $\mathcal Y^1_{2}$ are given by 
\begin{align*}
\setlength\arraycolsep{2pt}
\begin{pmatrix}
2 \\
{}
\end{pmatrix},
&&
\begin{pmatrix}
0 & 1 \\
2 &
\end{pmatrix},
&&
\begin{pmatrix}
0 & 2 \\
1 &
\end{pmatrix},
&&
\begin{pmatrix}
1 & 2 \\
0 &
\end{pmatrix},
&&
\begin{pmatrix}
0 & 1 & 2\\
1 & 2 &
\end{pmatrix},
&&
\begin{pmatrix}
0 & 1 & 2 \\
  &   &
\end{pmatrix}.
\end{align*}
The last symbol has defect $3$ whereas all the other symbols have defect $1$.
\end{ex}

The symbols can be used to classify the unipotent representations of the finite symplectic group, cf \cite{classical} Theorem 8.2.

\begin{theo}\label{ClassificationUnip} 
There is a natural bijection between $\mathcal Y^1_{\theta}$ and the set of equivalence classes of unipotent representations of $\mathrm{Sp}(2\theta,\mathbb F_q)$.
\end{theo}

If $S \in \mathcal Y^1_{\theta}$ we write $\rho_S$ for the associated unipotent representation of $\mathrm{Sp}(2\theta,\mathbb F_q)$. The classification is done so that the symbols
\begin{align*}
\setlength\arraycolsep{2pt}
\begin{pmatrix}
\theta \\
{}
\end{pmatrix},
&&
\begin{pmatrix}
0 & \ldots & \theta - 1 & \theta \\
1 & \ldots & \theta & 
\end{pmatrix},
\end{align*}
correspond respectively to the trivial and to the Steinberg representations. Let $S = (X,Y)$ be a symbol and let $k\geq 1$. A \textbf{$k$-hook} $h$ in $S$ is an integer $z\geq k$ such that $z\in X,z-k \not \in X$ or $z\in Y,z-k \not \in Y$. A \textbf{$k$-cohook} $c$ in $S$ is an integer $z\geq k$ such that $z \in X, z-k \not \in Y$ or $z\in Y, z-k \not \in X$. The integer $k$ is referred to as the \textbf{length} of the hook $h$ or the cohook $c$, and it is denoted $\ell(h)$ or $\ell(c)$. The \textbf{hook formula} gives an expression of $\dim(\rho_S)$ in terms of hooks and cohooks. 

\begin{prop}\label{HookFormula}
We have 
$$\dim(\rho_S) = q^{a(S)} \frac{\prod_{i=1}^{\theta} \left(q^{2i}-1\right)}{2^{b'(S)}\prod_{h}\left(q^{\ell(h)}-1\right)\prod_{c}\left(q^{\ell(c)}+1\right)},$$
where the products in the denominator run over all the hooks $h$ and all the cohooks $c$ in $S$, and the numbers $a(S)$ and $b'(S)$ are given by
\begin{align*}
a(S) = \sum_{\{s,t\}\subset S} \min(s,t) - \sum_{i\geq 1} \binom{\#S-2i}{2}, & & b'(S) = \left\lfloor\frac{\#S-1}{2}\right\rfloor - \#\left(X\cap Y\right).
\end{align*}
\end{prop}

For $\delta \geq 0$, we define the symbol 
$$S_{\delta} := 
\begin{pmatrix}
0 & \ldots & 2\delta \\
  &        & 
\end{pmatrix} \in \mathcal Y^1_{2\delta+1,\delta(\delta+1)}.$$

\begin{defi}
The \textbf{core} of a symbol $S \in \mathcal Y^1_{d,\theta}$ is defined by $\mathrm{core}(S) := S_{\delta}$ where $d = 2\delta + 1$. We say that $S$ is $\textbf{cuspidal}$ if $S = \mathrm{core}(S)$.
\end{defi}

\begin{rk}
In general, we have $\mathrm{rk}(\mathrm{core}(S)) \leq \mathrm{rk}(S)$ with equality if and only if $S$ is cuspidal. 
\end{rk}

The next theorem states that cuspidal unipotent representations correspond to cuspidal symbols.

\begin{theo}\label{CuspidalUnipotents}
The group $\mathrm{Sp}(2\theta,\mathbb F_q)$ admits a cuspidal unipotent representation if and only if $\theta = \delta(\delta+1)$ for some $\delta \geq 0$. When this is the case, the cuspidal unipotent representation is unique and given by $\rho_{S_{\delta}}$.
\end{theo}

The determination of the cuspidal unipotent representations leads to a description of the unipotent Harish-Chandra series. 

\begin{defi}
Let $\delta \geq 0$ such that $\theta = \delta(\delta+1) + a$ for some $a\geq 0$. We write 
$$L_{\delta} \simeq \mathrm{GL}(1,\mathbb F_q)^a \times \mathrm{Sp}(2\delta(\delta+1),\mathbb F_q)$$ 
for the block-diagonal Levi complement in $\mathrm{Sp}(2\theta,\mathbb F_q)$, with one middle block of size $2\delta(\delta+1)$ and other blocks of size $1$. We write $\rho_{\delta} := (\mathbf{1})^a \boxtimes \rho_{S_{\delta}}$, which is a cuspidal representation of $L_{\delta}$.
\end{defi}

\begin{prop}\label{UnipotentCuspidalSupport} 
Let $S\in \mathcal Y^{1}_{d,\theta}$. The cuspidal support of $\rho_S$ is $(L_{\delta},\rho_{\delta})$ where $d = 2\delta+1$.
\end{prop}

In particular, the defect of the symbol $S$ of rank $\theta$ classifies the unipotent Harish-Chandra series of $\mathrm{Sp}(2\theta,\mathbb F_p)$. If $\delta \geq 0$ is such that $\delta(\delta+1) \leq \theta$, we write $\mathcal E_{\delta}$ for the Harish-Chandra series determined by $(L_{\delta},\rho_{\delta})$. The previous proposition says that $\mathcal E_{\delta}$ is in bijection with $\mathcal Y^1_{2\delta+1,\theta}$. Representations in a given series $\mathcal E_{\delta}$ can also be labelled in an alternative way. 

\begin{defi}
A \textbf{partition} of an integer $n\geq 0$ is a sequence of positive integers $\lambda = (\lambda_1 \geq \ldots \geq \lambda_r)$ with $r\geq 0$ such that $n = \lambda_1 + \ldots + \lambda_r$. The integer $|\lambda| := n$ is called the length of $\lambda$. A \textbf{bipartition} of $n$ is a pair $(\lambda,\mu)$ of partitions such that $|\lambda| + |\mu| = n$. 
\end{defi}

Let $S \in \mathcal Y^1_{2\delta+1,\theta}$ where $\delta(\delta+1) \leq \theta$. Up to taking suitable shifts, we may consider representatives of $S$ and of $S_{\delta}$ whose rows have the same length. We obtain a bipartition $(\alpha,\beta)$ of $\theta - \delta(\delta+1)$ by subtracting component-wise the rows of $S_{\delta}$ from the rows of $S$, and reordering them decreasingly while ignoring the potential $0$ components. 

\begin{ex}
Recall the $6$ symbols of $\mathcal Y^1_2$ described in Exemple \ref{ExempleSymbols}. The associated bipartitions are respectively
\begin{align*}
((2),\emptyset), & & (\emptyset,(2)), & & ((1),(1)), & & ((1^2),\emptyset), & & (\emptyset,(1^2)) & & (\emptyset,\emptyset).
\end{align*}
The five symbols of defect $1$ correspond to bipartitions of $2$, and the last symbol of defect $3$ corresponds to the empty bipartition of $0$.
\end{ex}

\begin{prop}
The process described above defines a natural bijection between $\mathcal Y^1_{2\delta+1,\theta}$ and the set of bipartitions of $\theta - \delta(\delta+1)$.
\end{prop}

It is well-known that the set $\mathrm{Irr}(W_n)$ of equivalence classes of irreducible representations of the Coxeter group $W_n$ are in bijection with the set of bipartitions of $n$, cf. \cite{geck2} Section 5.5. Through this bijection, the bipartitions $((n),\emptyset)$ and $(\emptyset, (1^n))$ correspond respectively to the trivial and to the signature characters of $W_n$. 

\begin{corol}
Let $\theta, \delta\geq 0$ such that $\delta(\delta+1)\leq \theta$. There is a natural bijection between $\mathcal E_{\delta}$ and $\mathrm{Irr}(W_{\theta-\delta(\delta+1)})$.
\end{corol}

One may check that the trivial (resp. the sign) character of $W_{\theta}$ corresponds to the trivial (resp. the Steinberg) representation of $\mathrm{Sp}(2\theta,\mathbb F_q)$ in $\mathcal E_0$. Given a symbol $S \in \mathcal Y^1_{2\delta+1,\theta}$ and the corresponding bipartition $(\alpha,\beta)$ of $\theta - \delta(\delta+1)$, we will sometimes also write $\rho_{\delta,\alpha,\beta}$ instead of $\rho_S$. It turns out that the labelling of the unipotent representations of $\mathrm{Sp}(2\theta,\mathbb F_q)$ in terms of bipartitions is particularly well suited in order to compute Harish-Chandra inductions and restrictions. To this end, it is also convenient to identify partitions with their Young diagrams.

\begin{defi}
A \textbf{Young diagram} $T$ is a finite collection of boxes which are organized in rows of non-increasing lengths, justified on the left side. The size $|T|$ of a Young diagram is the number of boxes it contains. If $T$ and $T'$ are two Young diagrams such that $|T| = |T'|+1$, we say that $T$ is obtained from $T'$ by adding one box, or that $T'$ is obtained from $T$ by removing one box, if we can overlay $T'$ on the top of $T$ so that the difference consists of just one box. 
\end{defi}

\begin{ex}
Let us consider the following Young diagram of size $4$
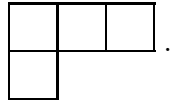
\begin{center}
\ydiagram{3,1}\;.
\end{center}
The Young diagrams which can be obtained by adding a box are given by
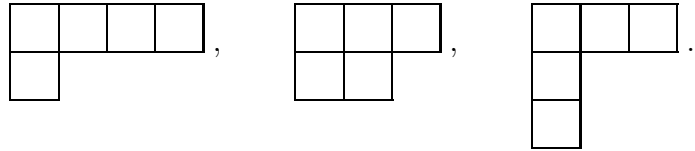
\begin{center}
\ydiagram{4,1}\;, \quad \quad \ydiagram{3,2}\;, \quad \quad \ydiagram{3,1,1}\;.
\end{center}
\end{ex}

Clearly the set of partitions of $n$ is in bijection with the set of Young diagrams of size $n$. Likewise, bipartitions of $n$ correspond to pairs of Young diagrams of combined sizes $n$. In the following, we will have to compute inductions of the following form
$$\mathrm R_{\mathrm{GL}(a,\mathbb F_q) \times \mathrm{Sp}(2\theta',\mathbb F_q)}^{\mathrm{Sp}(2\theta,\mathbb F_q)} \, \mathbf 1 \boxtimes \rho_{S'},$$
where $\theta = a + \theta'$ and $S' \in \mathcal Y^{1}_{d,\theta'}$ is a symbol. In particular we assume that $\theta' \geq \delta(\delta+1)$ where $d = 2\delta+1$. 

\begin{theo}\label{ComputeInduction}
Let $S' \in \mathcal Y^1_{d,\theta'}$ and $(\alpha',\beta')$ the associated bipartition of $\theta' - \delta(\delta+1)$. We have 
$$\mathrm R_{\mathrm{GL}(a,\mathbb F_q) \times \mathrm{Sp}(2\theta',\mathbb F_q)}^{\mathrm{Sp}(2\theta,\mathbb F_q)} \, \mathbf 1 \boxtimes \rho_{S'} = \sum_{\alpha,\beta} \rho_{\delta,\alpha,\beta},$$
where $(\alpha,\beta)$ runs over all the bipartitions of $\theta - \delta(\delta+1)$ such that for some $0 \leq d \leq a$, the Young diagram of $\alpha$ (resp. of $\beta$) can be obtained from the Young diagram of $\alpha'$ (resp. of $\beta'$) by adding a succession of $d$ boxes (resp. $a - d$ boxes), no two of them lying in the same column.
\end{theo}

This computation is a consequence of the Howlett-Lehrer comparison theorem \cite{howlett}, stating that induction in a finite group of Lie type can be computed inside a corresponding Weyl group. We then apply Pieri's rule for Coxeter groups of type $B_n$, see \cite{geck2} 6.1.9. We will see concrete exemples of such computations in the following sections. There is a similar rule in order to compute certain Harish-Chandra restrictions. We write ${}^*\mathrm R_{\mathrm{Sp}(2\theta',\mathbb F_q)}^{\mathrm{Sp}(2\theta,\mathbb F_q)}$ for the restriction to the symplectic part of the Harish-Chandra restriction functor from $\mathrm{Sp}(2\theta,\mathbb F_q)$ to the Levi complement $\mathrm{GL}(a,\mathbb F_q) \times \mathrm{Sp}(2\theta',\mathbb F_q)$. 

\begin{theo}\label{ComputeRestriction}
Let $S \in \mathcal Y^1_{d,\theta}$ and $(\alpha,\beta)$ the associated bipartition of $\theta - \delta(\delta+1)$. We have 
$${}^*\mathrm R_{\mathrm{Sp}(2\theta',\mathbb F_q)}^{\mathrm{Sp}(2\theta,\mathbb F_q)} \, \rho_{S} = \sum_{\alpha',\beta'} \rho_{\delta,\alpha',\beta'},$$
where $(\alpha',\beta')$ runs over all the bipartitions of $\theta' - \delta(\delta+1)$ such that for some $0 \leq d \leq a$, the Young diagram of $\alpha'$ (resp. of $\beta'$) can be obtained from the Young diagram of $\alpha$ (resp. of $\beta$) by removing a succession of $d$ boxes (resp. $a - d$ boxes), no two of them lying in the same column.
\end{theo}

\section{The cohomology of the Coxeter variety for the symplectic group}

\label{KnownResultsCoxeter}In this section we compute the cohomology of Coxeter varieties of finite symplectic groups, in terms of the classification of the unipotent characters that we recalled in Theorem \ref{ClassificationUnip}.

\begin{notation}
We write $X^{\theta} := X_{\emptyset}(\mathrm{cox})$ for the Coxeter variety attached to the symplectic group $\mathrm{Sp}(2\theta,\mathbb F_q)$, and $\mathrm{H}_c^{\bullet}(X^{\theta})$ instead of $\mathrm{H}_c^{\bullet}(X^{\theta}\otimes \mathbb F,\overline{\mathbb Q_{\ell}})$ where $\ell \not = p$.
\end{notation}

We first recall known facts on the cohomology of $X^{\theta}$ from Lusztig's work in \cite{cox}.

\begin{theo}
The following statements hold.
\begin{enumerate}[label=\upshape (\arabic*), topsep = 0pt]
\item The variety $X^{\theta}$ has dimension $\theta$ and is affine. The cohomology group $\mathrm{H}^i_c(X^{\theta})$ is zero unless $\theta\leq i \leq 2\theta$.
\item The Frobenius $F$ acts in a semisimple manner on the cohomology of $X^{\theta}$. 
\item The groups $\mathrm{H}^{2\theta-1}_c(X^{\theta})$ and $\mathrm{H}^{2\theta}_c(X^{\theta})$ are irreducible as $\mathrm{Sp}(2\theta,\mathbb F_q)$-representations, and the latter is the trivial representation. The Frobenius $F$ acts with eigenvalues respectively $q^{\theta-1}$ and $q^{\theta}$.
\item The group $\mathrm H^{\theta+i}_c(X^{\theta})$ for $0 \leq i \leq \theta-2$ is the direct sum of two eigenspaces of $F$, for the eigenvalues $q^{i}$ and $-q^{i+1}$. Each eigenspace is an irreducible unipotent representation of $\mathrm{Sp}(2\theta,\mathbb F_q)$.
\item The sum $\bigoplus_{i\geq 0} \mathrm H_c^i(X^{\theta})$ is multiplicity-free as a representation of $\mathrm{Sp}(2\theta,\mathbb F_q)$.
\end{enumerate}
\end{theo}

In other words, there exists a uniquely determined family of pairwise distinct symbols $S^\theta_0,\ldots ,S^{\theta}_{\theta}$ and $T^{\theta}_0,\ldots ,T^{\theta}_{\theta-2}$ in $\mathcal Y^{1}_{\theta}$ such that
\begin{align*}
\forall 0\leq i \leq \theta-2, & & \mathrm H^{\theta+i}_c(X^{\theta}) & \simeq \rho_{S^{\theta}_i} \oplus \rho_{T^{\theta}_i}, \\
\text{for } i = \theta-1,\theta, & & \mathrm H^{\theta+i}_c(X^{\theta}) & \simeq \rho_{S^{\theta}_{i}}.
\end{align*} 
The representation $\rho_{S^{\theta}_i}$ (resp. $\rho_{T^{\theta}_i}$) corresponds to the eigenspace of the Frobenius $F$ on $\bigoplus_{i\geq 0} \mathrm H_c^{\theta+i}(X^{\theta})$ attached to $q^i$ (resp. to $-q^{i+1}$). Moreover, we know that $\rho_{S^{\theta}_{\theta}}$ is the trivial representation, therefore 
$$\setlength\arraycolsep{2pt}
S^{\theta}_{\theta} = 
\begin{pmatrix}
\theta \\
{}
\end{pmatrix}.$$
Lusztig also gives a formula computing the dimension of the eigenspaces. Specializing to the case of the symplectic group, it reduces to the following statement.

\begin{prop}[\cite{cox}]
For $0\leq i \leq \theta$ we have 
$$\deg(\rho_{S^{\theta}_i}) = q^{(\theta-i)^2} \prod_{s=1}^{\theta-i} \frac{q^{s+i}-1}{q^s-1}\prod_{s=0}^{\theta-i-1} \frac{q^{s+i}+1}{q^s+1}.$$
For $0 \leq j \leq \theta-2$ we have 
$$\deg(\rho_{T^{\theta}_j}) = q^{(\theta-j-1)^2}\frac{(q^{\theta-1}-1)(q^\theta-1)}{2(q+1)}\prod_{s=1}^{\theta-j-2}\frac{q^{s+j}-1}{q^s-1}\prod_{s=2}^{\theta-j-1}\frac{q^{s+j}+1}{q^s+1}.$$
\end{prop}

Our goal in this section is to determine the symbols $S^{\theta}_i$ and $T^{\theta}_j$ explicitly. This is done in the following proposition.

\begin{prop}\label{CohomologyCoxeter}
For $0\leq i \leq \theta$ and $0\leq j \leq \theta-2$, we have 
\begin{align*}
\setlength\arraycolsep{2pt}
S^{\theta}_i = 
\begin{pmatrix}
0 & \ldots & \theta-i-1 & \,\,\,\theta \\
1 & \ldots & \theta-i & 
\end{pmatrix},
&&
T^{\theta}_j = 
\begin{pmatrix}
0 & \ldots & \theta-j-3 & \theta-j-2 & \theta-j-1 & \theta \\
1 & \ldots & \theta-j-2 & & & 
\end{pmatrix}.
\end{align*}
\end{prop}

\begin{rk}
In terms of bipartitions, $S^{\theta}_i$ corresponds to $((i),(1^{\theta-i}))$ and $T^{\theta}_j$ corresponds to $((j),(1^{\theta-2-j}))$.
\end{rk}

We note that the statement is coherent with the two dimension formulae that we provided earlier. That is, the degree of $\rho_{S^{\theta}_i}$ (resp. of $\rho_{T^{\theta}_j}$) computed with the hook formula of Proposition \ref{HookFormula}, agrees with the dimension of the eigenspace of $q^i$ (resp. of $-q^{j+1}$) in the cohomology of $X^{\theta}$ as given in the previous paragraph.

\begin{proof} We use induction on $\theta\geq 0$. Since we already know that $S^{\theta}_{\theta}$ is the symbol corresponding to the trivial representation, the proposition is proved for $\theta=0$. Thus we may assume $\theta\geq 1$. We consider the block diagonal Levi complement $L \simeq \mathrm{GL}(1,\mathbb F_q) \times \mathrm{Sp}(2(\theta-1),\mathbb F_q)$, and we write ${}^*\mathrm R_{\theta-1}^{\theta}$ for the restriction to $\mathrm{Sp}(2(\theta-1),\mathbb F_q)$ of the Harish-Chandra restriction from $\mathrm{Sp}(2\theta,\mathbb F_q)$ to $L$. According to \cite{cox} Corollary 2.10, for all $0\leq i \leq \theta$ we have an $\mathrm{Sp}(2(\theta-1),\mathbb F_q)\times \langle F \rangle$-equivariant isomorphism 
\begin{equation}\label{eq}
{}^*\mathrm R_{\theta-1}^{\theta} \left(\mathrm H^{\theta+i}_c(X^{\theta})\right) \simeq \mathrm H^{\theta-1+i}_c(X^{\theta-1}) \oplus \mathrm H^{\theta-1+(i-1)}_c(X^{\theta-1})(-1)\tag{$*$}.
\end{equation}
The right-hand side can be computed by induction hypothesis whereas the left-hand side can be computed using Theorem \ref{ComputeRestriction}. We fix $0 \leq i \leq \theta - 1$ and $0 \leq j \leq \theta - 2$. We denote by $(\delta,\alpha,\beta)$ and $(\nu,\gamma,\delta)$ the alternate labelling by bipartitions of the representations $\rho_{S^{\theta}_i}$ and $\rho_{T^{\theta}_j}$ respectively. Recall that the restriction ${}^*\mathrm R_{\theta-1}^{\theta}\,\rho_{\delta,\alpha,\beta}$ is the sum of all the representations $\rho_{\delta,\alpha',\beta'}$ where $(\alpha',\beta')$ is obtained from $(\alpha,\beta)$ by removing a single box in one of their Young diagrams. The similar description also holds for ${}^*\mathrm R_{\theta-1}^{\theta}\,\rho_{\nu,\gamma,\delta}$. \\
First we determine $S^{\theta}_i$ by identifying the $q^i$-eigenspace of the Frobenius in \eqref{eq}. We distinguish different cases depending on the values of $\theta$ and $i$.
\begin{enumerate}[label={--},noitemsep,topsep=0pt]
\item \textbf{Case }$\mathbf{\theta=1}$. In this case $i = 0$. The right-hand side of \eqref{eq} is $\rho_{S^0_0} \simeq \rho_{0,\emptyset,\emptyset}$ with eigenvalue $1$. Thus, $\delta = 0$ and the bipartition $(\alpha,\beta)$ consists of a single box. Therefore $(\alpha,\beta) = ((1),\emptyset)$ or $(\emptyset,(1))$. By Theorem \ref{KnownResultsCoxeter}, we know that $\rho_{0,\alpha,\beta}$ has degree $q$. This forces $(\alpha,\beta) = (\emptyset,(1))$.
\item \textbf{Case }$\mathbf{\theta=2}$ \textbf{ and }$\mathbf{i=0}$. The eigenspace attached to $1$ on the right-hand side of \eqref{eq} is $\rho_{S^{1}_0} \simeq \rho_{0,\emptyset,(1)}$. Thus, $\delta = 0$ and there is a single removable box in the bipartition $(\alpha,\beta)$. When we remove it, we obtain $(\emptyset,(1))$. Therefore, $(\alpha,\beta) = (\emptyset,(2))$ or $(\emptyset,(1^2))$. By Theorem \ref{KnownResultsCoxeter}, we know that $\rho_{0,\alpha,\beta}$ has degree $q^4$, thus $(\alpha,\beta) = (\emptyset,(1^2))$.

\item \textbf{Case }$\mathbf{\theta>2}$\textbf{ and }$\mathbf{i=0}$. The eigenspace attached to $1$ on the right-hand side of \eqref{eq} is $\rho_{S^{\theta-1}_0} \simeq \rho_{0,\emptyset,(1^{\theta-1})}$. Thus, $\delta = 0$ and there is a single removable box in the bipartition $(\alpha,\beta)$. When we remove it, we obtain $(\emptyset,(1^{\theta-1}))$. The only such bipartition is $(\alpha,\beta) = (\emptyset,(1^{\theta}))$.
\item \textbf{Case }$\mathbf{\theta>2}$\textbf{ and }$\mathbf{1\leq i \leq k-1}$. The eigenspace attached to $p^i$ on the right-hand side of \eqref{eq} is $\rho_{S^{\theta-1}_i} \oplus \rho_{S^{\theta-1}_{i-1}} \simeq \rho_{0,(i),(1^{\theta-1-i})} \oplus \rho_{0,(i-1),(1^{\theta-i})}$. Thus, $\delta = 0$ and there are exactly two removable boxes in the bipartition $(\alpha,\beta)$. When we remove one of them, we obtain either $((i),(1^{\theta-1-i}))$ or $((i-1),(1^{\theta-i}))$. The only such bipartition is $(\alpha,\beta) = ((i),(1^{\theta-i}))$.
\end{enumerate}

It remains to determine $T^{\theta}_j$ for $0\leq j \leq \theta-2$. 

\begin{enumerate}[label={--},noitemsep,topsep=0pt]
\item \textbf{Case }$\mathbf{\theta=2}$. The eigenspace attached to $-p$ on the right-hand side of \eqref{eq} is $0$. Thus, the symbol $T^2_0 \in \mathcal Y^1_2$ has no hook at all, implying that it is cuspidal. Since $\mathrm{Sp}(4,\mathbb F_q)$ admits only one unipotent cuspidal representation, we deduce that $\nu = 1$ and $(\gamma,\delta) = (\emptyset,\emptyset)$.
\item \textbf{Case }$\mathbf{k=3}$. First when $j=0$, the eigenspace attached to $-p$ on the right-hand side of \eqref{eq} is $\rho_{T^{2}_0} \simeq \rho_{1,\emptyset,\emptyset}$. Thus, $\nu = 1$ and there is a single box in the bipartition $(\gamma,\delta)$. Therefore $(\gamma,\delta) = ((1),\emptyset)$ or $(\emptyset,(1))$. By Theorem \ref{KnownResultsCoxeter}, we know that $\rho_{1,\gamma,\delta}$ has degree $q^4\frac{(q^2-1)(q^3-1)}{2(q+1)}$, thus $(\gamma,\delta) = (\emptyset,(1))$.\\
Then when $j=1$, the eigenspace attached to $-p^2$ on the right-hand side of \eqref{eq} is $\rho_{T^{2}_0} \simeq \rho_{1,\emptyset,\emptyset}$. Thus, $\nu = 1$ and as in the case $j=0$ we have $(\gamma,\delta) = ((1),\emptyset)$ or $(\emptyset,(1))$. We can deduce that it is equal to the former by comparing the dimensions or by using the fact that the symbols $T_j^{\theta}$ are pairwise distinct.
\item \textbf{Case }$\mathbf{\theta=4}$\textbf{ and }$\mathbf{j=0}$. The eigenspace attached to $-p$ on the right-hand side of \eqref{eq} is $\rho_{T^{3}_0} \simeq \rho_{1,\emptyset,(1)}$. Thus, $\nu = 1$ and there is a single removable box in the bipartition $(\gamma,\delta)$. When we remove it, we obtain $(\emptyset,(1))$. Therefore, $(\gamma,\delta) = (\emptyset,(2))$ or $(\emptyset,(1^2))$. By Theorem \ref{KnownResultsCoxeter}, we know that $\rho_{1,\gamma,\delta}$ has degree $q^9\frac{(q^3-1)(q^4-1)}{2(q+1)}$, thus $(\gamma,\delta) = (\emptyset,(1^2))$.
\item \textbf{Case }$\mathbf{\theta>4}$\textbf{ and }$\mathbf{j=0}$. The eigenspace attached to $-p$ on the right-hand side of \eqref{eq} is $\rho_{T^{\theta-1}_0} \simeq \rho_{1,\emptyset,(1^{\theta-3})}$. Thus, $\nu = 1$ and there is a single removable box in the bipartition $(\gamma,\delta)$. When we remove it, we obtain $(\emptyset,(1^{\theta-3}))$. The only such bipartition is $(\gamma,\delta) = (\emptyset,(1^{\theta-2}))$.
\item \textbf{Case }$\mathbf{\theta=4}$\textbf{ and }$\mathbf{j=\theta-2}$. The eigenspace attached to $-p^{3}$ on the right-hand side of \eqref{eq} is $\rho_{T^{3}_{1}} \simeq \rho_{1,(1),\emptyset}$. Thus, $\nu = 1$ and there is a single removable box in the bipartition $(\gamma,\delta)$. When we remove it, we obtain $((1),\emptyset)$. Therefore, $(\gamma,\delta) = ((2),\emptyset)$ or $((1^2),\emptyset)$. By Theorem \ref{KnownResultsCoxeter}, we know that $\rho_{1,\gamma,\delta}$ has degree $q\frac{(q^3-1)(q^4-1)}{2(q+1)}$, thus $(\gamma,\delta) = ((2),\emptyset)$.
\item \textbf{Case }$\mathbf{\theta>4}$\textbf{ and }$\mathbf{j = \theta-2}$. The eigenspace attached to $-p^{\theta-1}$ on the right-hand side of \eqref{eq} is $\rho_{T^{\theta-1}_{\theta-3}} \simeq \rho_{1,(\theta-3),\emptyset}$. Thus, $\nu = 1$ and there is a single removable box in the bipartition $(\gamma,\delta)$. When we remove it, we obtain $((\theta-3),\emptyset)$. The only such bipartition is $(\gamma,\delta) = ((\theta-2),\emptyset)$.
\item \textbf{Case }$\mathbf{1\leq j \leq \theta-3}$. The eigenspace attached to $-p^{j+1}$ on the right-hand side of \eqref{eq} is $\rho_{T^{\theta-1}_{j}} \oplus \rho_{T^{\theta-1}_{j-1}} \simeq \rho_{1,(j),(1^{\theta-3-j})}\oplus \rho_{1,(j-1),(1^{\theta-2-j})}$. Thus, $\nu = 1$ and there are exactly two removable boxes in the bipartition $(\gamma,\delta)$. When we remove one of them, we obtain either $((j),(1^{\theta-3-j}))$ or $((j-1),(1^{\theta-2-j}))$. The only such bipartition is $(\gamma,\delta) = ((j),(1^{\theta-2-j}))$.
\end{enumerate}
\end{proof}

\section{The cohomology of $S_{\theta}$}

The last three sections of the paper are devoted to proving the main theorem below, which describes the cohomology of the variety $S_{\theta}$. Since $S_0$ is a point and $S_1 \simeq \mathbb P^1$, the cases $\theta = 0$ or $1$ are trivial. \textbf{From now and up to the end of the paper, we assume that $\theta \geq 2$}.

\begin{theo}\label{MainTheorem}
The following statements hold.
\begin{enumerate}[label=\upshape (\arabic*), topsep = 0pt]
		\item All the cohomology groups of $S_{\theta}$ of odd degree vanish.
		\item For $0 \leq i \leq \theta$, we have an $\mathrm{Sp}(2\theta,\mathbb F_q)$-equivariant isomorphism
		$$\mathrm H^{2i}(S_{\theta}) \simeq \bigoplus_{s=0}^{\min(i,\theta-i)} \rho_{\scaleto{
		\begin{pmatrix}
		s & \theta+1-s \\
		0 & {} 
		\end{pmatrix}}{25pt}} \oplus \bigoplus_{s = 0}^{\min(i-1,\theta-1-i)} \rho_{\scaleto{
		\begin{pmatrix}
		0 & s + 1 & \theta-s \\
		{} & {} & {} 
		\end{pmatrix}}{25pt}}.$$
		The Frobenius acts like multiplication by $q^i$ on the first summand, and multiplication by $-q^i$ on the second summand.
	\end{enumerate}
\end{theo}

\begin{rks} Let us make a few comments.
\begin{enumerate}[label={--},noitemsep,topsep=0pt]
\item We may rewrite the formula in terms of the alternate labelling of the unipotent representations of $\mathrm{Sp}(2\theta,\mathbb F_p)$. We obtain
$$\mathrm H^{2i}(S_{\theta}) \simeq \bigoplus_{s=0}^{\min(i,\theta-i)} \rho_{0,(\theta-s,s),\emptyset} \oplus \bigoplus_{s = 0}^{\min(i-1,\theta-1-i)} \rho_{1,(\theta-2-s,s),\emptyset}.$$
\item A unipotent cuspidal representation occurs in the cohomology of $S_{\theta}$ only in the cases $\theta = 0$ and $\theta = 2$. When $\theta = 0$ it corresponds to $\mathrm H^0(S_{0})$ which is trivial. When $\theta = 2$ it occurs in $\mathrm H^2(S_{2})$ with the eigenvalue $-p$. All the representations occuring in the cohomology of $S_{\theta}$ have cuspidal support given by one of these two cuspidal unipotent representations.
\item Even though $S_{\theta}$ has isolated singularities for $\theta \geq 2$, its cohomology looks like the cohomology of a smooth projective variety, in so that it satisfies Poincaré duality, hard Lefschetz and purity of the Frobenius action. 
\end{enumerate}
\end{rks}

In order to compute the cohomology of $S_{\theta}$, we use the stratification by classical Deligne-Lusztig varieties which we recalled in Proposition \ref{Stratification}, and we analyze the associated spectral sequence. It is given in its first page by  
\begin{equation}\label{spectral}
E^{\theta',i}_1 = \mathrm H^{\theta'+i}_c(X_{I_{\theta'}}(w_{\theta'})) \implies \mathrm H^{\theta'+i}(S_{\theta}).
\tag{$E$}
\end{equation}
Let us first determine each term explicitely. By Proposition \ref{TransitivityIdentity} and using the notations introduced there, we have an isomorphism 
$$X_{I_{\theta'}}(w_{\theta'}) \simeq \mathrm{Sp}(2\theta,\mathbb F_q)/U_{K_{\theta'}} \times_{L_{K_{\theta'}}} X_{I_{\theta'}}^{\mathbf L_{K_{\theta'}}}(w_{\theta'}).$$
Taking cohomology, this identity translates into some Harish-Chandra induction 
$$\mathrm H^{\bullet}_c(X_{I_{\theta'}}(w_{\theta'})) \simeq \mathrm R_{L_{K_{\theta'}}}^{\mathrm{Sp}(2\theta,\mathbb F_q)} \, \mathrm H^{\bullet}_c(X_{I_{\theta'}}^{\mathbf L_{K_{\theta'}}}(w_{\theta'})).$$
The Deligne-Lusztig variety $X_{I_{\theta'}}^{\mathbf L_{K_{\theta'}}}(w_{\theta'})$ for the Levi complement $L_{K_{\theta'}} \simeq \mathrm{GL}(\theta-\theta',\mathbb F_q) \times \mathrm{Sp}(2\theta',\mathbb F_q)$ is isomorphic to the Coxeter variety $X^{\theta'}$ with the $\mathrm{GL}$-part acting trivially. Thus, the terms $E_1^{\theta',i}$ are the Harish-Chandra inductions of the cohomology groups of the Coxeter varieties $X^{\theta'}$, which we have determined in Proposition \ref{CohomologyCoxeter}. Let us compute these inductions explicitely.

\begin{lem}\label{FirstTerms}
Let $0 \leq i \leq \theta' \leq \theta$. We have $E_1^{\theta',i} = A^{\theta',i} \oplus B^{\theta',i}$ with
\begin{align*}
A^{\theta',i} = \bigoplus_{\alpha,\beta} \rho_{0,\alpha,\beta}, & & B^{\theta',i} = \bigoplus_{\gamma,\delta} \rho_{1,\gamma,\delta},
\end{align*}
where $(\alpha,\beta)$ runs over all the bipartitions of $\theta$ such that, for some $0 \leq d \leq \theta - \theta'$, we have 
$$\begin{cases}
\alpha = (i + d - s, s) \text{ for some } 0 \leq s \leq \min(d,i),\\
\beta = (\theta-\theta'-d, 1^{\theta'-i})\text{ or } (\theta-\theta'-d+1,1^{\theta'-1-i}),
\end{cases}$$
and if $i \leq \theta'-2$, $(\gamma,\delta)$ runs over all the bipartitions of $\theta - 2$ such that, for some $0 \leq d \leq \theta - \theta'$, we have
$$\begin{cases}
\gamma = (i + d - s, s) \text{ for some } 0 \leq s \leq \min(d,i),\\
\delta = (\theta-\theta'-d, 1^{\theta'-2-i}) \text{ or } (\theta-\theta'-d+1,1^{\theta'-3-i}).
\end{cases}$$
The summand $A^{\theta',i}$ (resp. $B^{\theta',i}$) is the eigenspace of the Frobenius for the eigenvalue $q^i$ (resp. the eigenvalue $-q^{i+1}$).
\end{lem}

\begin{rk}
In particular, for each $d$ there are at most two possibilities for $\beta$ and $\delta$. To remove ambiguity, let us point out that there are just two situations, possibly overlapping, where the given possibilities for $\beta$ (resp. $\delta$) actually coincide, that is 
\begin{enumerate}[label={--},noitemsep,topsep=0pt]
\item if $i = \theta'$ (resp. $i = \theta' - 2$), in which case $\beta = (\theta - \theta' - d)$ (resp. $\delta = (\theta - \theta' - d)$),
\item if $d = \theta - \theta'$, in which case $\beta = (1^{\theta' - i})$ and $\delta = (1^{\theta'-2-i})$.
\end{enumerate}
\end{rk}

\begin{proof}
Using Pieri's rule for Coxeter groups of type $B_n$ as we recalled in Proposition \ref{ComputeInduction}, we must decompose the Harish-Chandra inductions
\begin{align*}
R_{\mathrm{GL}(\theta-\theta',\mathbb F_q) \times \mathrm{Sp}(2\theta',\mathbb F_q)}^{\mathrm{Sp}(2\theta,\mathbb F_q)} \, \mathbf 1 \boxtimes \rho_{S_i^{\theta'}}, & & R_{\mathrm{GL}(\theta-\theta',\mathbb F_q) \times \mathrm{Sp}(2\theta',\mathbb F_q)}^{\mathrm{Sp}(2\theta,\mathbb F_q)} \, \mathbf 1 \boxtimes \rho_{T_j^{\theta'}},
\end{align*}
where $S_i^{\theta'}$ and $T_j^{\theta'}$ are the symbols determined in Proposition \ref{CohomologyCoxeter}. The Young diagrams of the bipartitions associated to these symbols have the form 
\begin{center}
\ydiagram{3}*[\ldots]{3+1}*{4+1} \quad , \quad \ydiagram{1,1,1,0,1}*[\vdots]{0,0,0,1}
\end{center}
The problem is to determine all the pairs of Young diagrams one may obtain after adding a succession of $\theta-\theta'$ boxes to the pair of diagrams above, with no two boxes in the same column. This computation has already been done in \cite{mullerBT} Section 5, and leads to the claimed formula.
\end{proof}

\begin{corol}\label{EvenWeights}
The spectral sequence \eqref{spectral} degenerates in the second page and the resulting filtration on the abutment splits. The weights of the eigenvalues of the Frobenius action on $\mathrm H^{k}(S_{\theta})$ are  even and at most equal to $2\lfloor \frac{k}{2} \rfloor$.
\end{corol}

Recall that an eigenvalue $\alpha \in \overline{\mathbb Q_{\ell}}$ of the Frobenius on the cohomology of a variety defined over $\mathbb F_q$ is said to be of weight $w\in \mathbb Z$ if, for any isomorphism $\iota:\overline{\mathbb Q_{\ell}} \simeq \mathbb C$, we have $|\iota(\alpha)| = q^{\frac{w}{2}}$. 

\begin{proof}
By the previous lemma, two terms of the sequence lying on different rows have no common eigenvalues for the Frobenius morphism. The Frobenius equivariance of the differentials force them to vanish in pages after $E_1$. Therefore the spectral sequence degenerates in the second page. For any $0 \leq k \leq 2\theta$, we deduce the existence of a filtration $\mathrm{Fil}^{\bullet}$ on $\mathrm H^k(S_{\theta})$ such that the graded pieces $\mathrm{Gr}^p := \mathrm{Fil}^p / \mathrm{Fil}^{p+1}$ are isomorphic to $E_2^{p,k-p}$. In particular the non-zero graded pieces are concentrated in degree $k \leq 2p \leq 2\min(k,\theta)$. Each term $E_2^{p,k-p}$ is a subquotient of $E_1^{p,k-p}$. The Frobenius acts semisimply on the latter space with at most $2$ distinct eigenvalues, which are $q^{k-p}$, and $-q^{k-p+1}$ if $k+2 \leq 2p$. Since there is no common eigenvalue of the Frobenius in two different graded pieces, the filtration splits and the Frobenius acts semi-simply on $\mathrm H^k(S_{\theta})$. Moreover, the eigenvalues form a subset of $\{q^i, -q^{j} \,|\, 0 \leq i,j \leq \lfloor\frac{k}{2}\rfloor\}$.
\end{proof}

Analyzing the $\mathrm{Sp}(2\theta,\mathbb F_q)$-action, we may decompose each term $E_1^{\theta',i}$ in the following way,
\begin{align*}
A^{\theta',i} = A^{\theta',i}_0 \oplus A^{\theta',i}_1, & & B^{\theta',i} = B^{\theta',i}_0 \oplus B^{\theta',i}_1,
\end{align*}
where for $\epsilon = 0,1$ the term $A^{\theta',i}_{\epsilon}$ is the sum of all the irreducible components $\rho_{0,\alpha,\beta}$ with the partition $\beta$, written as $\beta = (\beta_1 \geq \ldots \geq \beta_r > 0)$, satisfies $r = \theta' + \epsilon - i$, and if $i \leq \theta'-2$ the term $B^{\theta',i}_{\epsilon}$ is the sum of all the irreducible components $\rho_{1,\gamma,\delta}$ with the partition $\delta$, written as $\delta = (\delta_1 \geq \ldots \geq \delta_s > 0)$, satisfies $s = \theta' - 2 + \epsilon - i$. We observe that $A^{\theta,i}_1 = B^{\theta,i}_1 = 0$, and for $0 \leq i \leq \theta' < \theta$ (resp. $0 \leq j + 2 \leq \theta' < \theta$) we have isomorphisms $A^{\theta',i}_1 \simeq A^{\theta'+1,i}_0 $ and $B^{\theta',j}_1 \simeq B^{\theta'+1,j}_0$. Consider a differential 
$$d^{\theta',i}: E_1^{\theta',i} \to E_1^{\theta'+1,i}.$$
Since the differentials are Frobenius equivariant, they decompose as a sum $d^{\theta',i} = d_{A}^{\theta',i} \oplus d_{B}^{\theta',i}$ where 
\begin{align*}
d_{A}^{\theta',i}: A^{\theta',i} \to A^{\theta'+1,i}, & & d_{B}^{\theta',i}: B^{\theta',i} \to B^{\theta'+1,i}.
\end{align*}
The $\mathrm{Sp}(2\theta,\mathbb F_q)$-equivariance then forces
\begin{align*}
\mathrm{Im}(d_{A}^{\theta'-1,i}) \subset A_0^{\theta',i} \subset \mathrm{Ker}(d_{A}^{\theta',i}), & & \mathrm{Im}(d_{B}^{\theta'-1,i}) \subset B_0^{\theta',i} \subset \mathrm{Ker}(d_{B}^{\theta',i}).
\end{align*}

In order to help visualize the situation, the page $E_1$ is drawn in Figure \ref{Figure1}. Since the sequence degenerates in $E_2$ and the resulting filtration splits, it is clear that for all $0 \leq i \leq \theta$, the cohomology group $\mathrm H^{2i}(S_{\theta})$ contains $A_0^{i,i} \oplus B^{i+1,i-1}_0$ (the term $B^{i+1,i-1}_0$ is non-zero if and only if $0 < i < \theta$). We point out that in Theorem \ref{MainTheorem}, (2) can be rephrased as $\mathrm H^{2i}(S_{\theta}) \simeq A_0^{i,i} \oplus B^{i+1,i-1}_0$ for all $0 \leq i \leq \theta$ . 

\begin{lem}\label{Equivalence}
In the statement of Theorem \ref{MainTheorem}, (1) is equivalent to (2).
\end{lem}

\begin{proof}
Let us fix $0\leq i \leq \theta' \leq \theta$. The equality $\mathrm{Im}(d_{A}^{\theta'-1,i}) = A_0^{\theta',i}$ (resp. $\mathrm{Im}(d_{B}^{\theta'-1,i}) = B_0^{\theta',i}$) is equivalent to $\mathrm{Ker}(d_{A}^{\theta'-1,i}) = A_0^{\theta'-1,1}$ (resp. $\mathrm{Ker}(d_{B}^{\theta'-1,i}) = B_0^{\theta'-1,i}$). Thus, the vanishing of a term $E_2^{\theta',i} = \mathrm{Ker}(d_{A}^{\theta',i})/\mathrm{Im}(d_{A}^{\theta'-1,i}) \oplus \mathrm{Ker}(d_{B}^{\theta',i})/\mathrm{Im}(d_{B}^{\theta'-1,i})$ is equivalent to the equalities 
\begin{align*}
\mathrm{Ker}(d_{A}^{\theta'-1,i}) = A_0^{\theta'-1,i}, & &  \mathrm{Im}(d_A^{\theta',i}) = A_0^{\theta'+1,i}, \\ 
\mathrm{Ker}(d_{B}^{\theta'-1,i}) = B_0^{\theta'-1,i}, & &  \mathrm{Im}(d_B^{\theta',i}) = B_0^{\theta'+1,i}.
\end{align*}
The lemma follows easily.
\end{proof}

\section{Desingularization of $S_{\theta}$ and purity of the Frobenius action on cohomology}
\label{Section5}

In order to make out how the spectral sequence simplifies in the second page, it is necessary to get more information on the expected weights of the Frobenius on the abutment. To this end, we introduce the blow-up $\pi:S_{\theta}'\to S_{\theta}$ at its singular points. We denote by $E := \pi^{-1}(Z)$ the exceptional divisor, where in the notations of Proposition \ref{RationalPointsStheta}, $Z = X_I(\mathrm{id})$ is the singular locus of $S_{\theta}$. Recall that $\dim Z = 0$. In this section, we prove the following Proposition.

\begin{prop}\label{SmoothBlowup}
The varieties $S_{\theta}'$ and $E$ are smooth.
\end{prop}

Since $S_{\theta}'\setminus E$ is isomorphic to the smooth locus of $S_{\theta}$, it is enough to prove that the blow-up is smooth at points of the exceptional divisor. To this end, we exhibit a certain affine neighborhood of any singular point of $S_{\theta}$. Recall the symplectic space $V$ introduced in the first section. Let $L(\theta)$ denote the Lagrangian Grassmannian of $V$, ie. the smooth projective variety defined over $\mathbb F_q$ whose $k$-rational points, for any field extension $k/\mathbb F_q$, correspond to subspaces of $U\subset V_k$ such that $U^{\perp} = U$. Recall from Proposition \ref{RationalPointsStheta} that $S_{\theta}$ is the closed subvariety of $L(\theta)$ consisting of those $U$ such that $U\cap \tau(U) \overset{\leq 1}{\subset} U$.\\
Let $\mathrm{Sym}_{\theta} \simeq \mathbb A^{\frac{\theta(\theta+1)}{2}}$ denote the variety of $\theta\times\theta$ symmetric matrices over $\mathbb F_q$. Let $V_{\theta}$ denote the closed subvariety consisting of all $M\in \mathrm{Sym}_{\theta}$ such that $M^{(q)} - M$ has rank at most one. 

\begin{lem}
Any singular point $U \in S_{\theta}$ has an open affine neighborhood isomorphic to $V_{\theta}$.
\end{lem}

\begin{proof}
If $U \in L(\theta)$ is a closed point defined over a finite extension $k/\mathbb F_q$, one may choose an isotropic supplement $U'$ so that we have a decomposition $V_k = U \oplus U'$. The symplectic pairing induces an identification between $U'$ and the $k$-linear dual of $U$. We may consider the affine variety $\mathrm{Hom}(U,U')$ defined over $\mathrm{Spec}(k)$, and the subvariety $\mathrm{Hom}(U,U')^{\mathrm{sym}}$ consisting of morphisms $\varphi \in \mathrm{Hom}(U,U')$ such that $\varphi^* = \varphi$, where $\varphi^*:U'^* \simeq U^{**} \simeq U \to U^*$ is the dual of $\varphi$. According to \cite{grassmann} Lemma 2.8, we have $\varphi \in \mathrm{Hom}(U,U')^{\mathrm{sym}}$ if and only $\Gamma_{\varphi}\in L(\theta)$ where $\Gamma_{\varphi}$ denotes the graph of $\varphi$. The assignment $\varphi \mapsto \Gamma_{\varphi}$ defines an open immersion $\mathrm{Hom}(U,U')^{\mathrm{sym}} \hookrightarrow L(\theta)$, identifying the former with an open affine neighborhood of $U$. We have an identification $\mathrm{Hom}(U,U')^{\mathrm{sym}} \simeq \mathrm{Sym}_{\theta} \otimes k$ upon fixing a basis of $U$ and equipping $U'$ with the dual basis.\\
Now assume that $U \in S_{\theta}$ is a singular point, equivalently an $\mathbb F_q$-rational point of $S_{\theta}$. Let us fix a basis $(f_i)_{1\leq i \leq \theta}$ of $U$ and equip $U'$ with the dual basis $(f_i')_{1\leq i \leq \theta}$. Let $M \in \mathrm{Sym}_{\theta}$. The graphs $\Gamma_M$ and $\Gamma_{M^{(q)}}$ are generated by the vectors 
\begin{align*}
g_j := f_j + \sum_{i=1}^{\theta} M_{ij}f_i', & & g_j^{(q)} = f_j + \sum_{i=1}^{\theta} M_{ij}^qf_i',
\end{align*}
respectively. A direct computation shows that the intersection $\Gamma_M \cap \Gamma_{M^{(q)}}$ is isomorphic to $\mathrm{Ker}(M^{(q)}-M)$. Since $U$ is defined over $\mathbb F_q$, the vectors $f_j$ and $f_j'$ are fixed by $\tau$, thus we have $\tau(\Gamma_M) = \Gamma_{M^{(q)}}$. It follows that $M \in S_{\theta} \cap \mathrm{Sym}_{\theta}$ if and only if $\dim \mathrm{Ker}(M^{(q)}-M) \geq \theta - 1$, ie. if and only if $M \in V_{\theta}$. 
\end{proof}

Let $V_{\theta}' \subset \mathrm{Sym}_{\theta}$ denote the subvariety of symmetric matrices $M$ of rank at most $1$. The variety $V_{\theta}'$ is known as a symmetric determinantal variety. It admits a single singular point corresponding to $M=0$.

\begin{prop}\label{BlowupDeterminantalVariety}
The blow-up of $V_{\theta}'$ at the point $M=0$ is smooth with smooth exceptional divisor.
\end{prop}

\begin{proof}
This is Theorem B of \cite{blowup} where, in their notations, we have $R_0 = \mathbb F_q$, $m = \theta$ and $r=2$.
\end{proof}

\begin{proof}[Proof of Proposition \ref{SmoothBlowup}]
The variety $V_{\theta}$ contains the $\mathbb F_q$-points of $\mathrm{Sym}_{\theta}$ and is singular precisely at them. By flat base change, the blow-up of $V_{\theta}$ at these singular points is an open neighborhood in $S_{\theta}'$ of the exceptional divisors above them. Thus it is enough to check smoothness of the blow-up of $V_{\theta}$. Moreover, the Lang map $M \mapsto M^{(q)} - M$ defines a finite étale cover $V_{\theta} \to V_{\theta}'$. By flat base change again, we are reduced to Proposition \ref{BlowupDeterminantalVariety}.
\end{proof}

Proposition \ref{SmoothBlowup} has the following consequence regarding the cohomology of $S_{\theta}$. 

\begin{corol}\label{Purity}
For $0 \leq k \leq 2\theta$, the Frobenius action on $\mathrm H^k(S_{\theta})$ is pure of weight $2\lfloor \frac{k}{2} \rfloor$.
\end{corol}

\begin{proof}
By \cite[\href{https://stacks.math.columbia.edu/tag/0EW3}{Lemma 0EW3}]{stacks-project}, there is a long exact sequence 
$$\ldots \to \mathrm H^k(S_{\theta}) \to \mathrm H^k(S_{\theta}') \oplus \mathrm H^k(Z) \to \mathrm H^k(E) \to \mathrm H^{k+1}(S_{\theta}) \to \ldots $$
Since $S_{\theta}'$ and $E$ are projective and smooth, the Frobenius action on their $k$-th cohomology group is pure of weight $k$. For any $0 \leq j \leq \theta$, the maps $\mathrm H^{2j-1}(E) \to \mathrm H^{2j}(S_{\theta})$ and $\mathrm H^{2j+1}(S_{\theta}) \to \mathrm H^{2j+1}(S_{\theta}')$ vanish since the cohomology of $S_{\theta}$ only contains even weights by Corollary \ref{EvenWeights}. Thus, we have in fact exact sequences 
$$0 \to \mathrm H^{2j}(S_{\theta}) \to \mathrm H^{2j}(S_{\theta}') \oplus \mathrm H^{2j}(Z) \to \mathrm H^{2j}(E) \to \mathrm H^{2j+1}(S_{\theta}) \to 0.$$
It follows that $\mathrm H^{2j}(S_{\theta})$ and $\mathrm H^{2j+1}(S_{\theta})$ are pure of weight $2j$.
\end{proof}

\section{Intersection cohomology of $S_{\theta}$}

Let $U := S_{\theta} \setminus Z$ denote the smooth locus of $S_{\theta}$. Let $j:U \hookrightarrow S_{\theta}$ be the open immersion. The \textbf{intersection complex} of $S_{\theta}$ is the intermediate image $j_{!*}\overline{\mathbb Q_{\ell}}$. We write $\mathrm{IH}^{\bullet}(S_{\theta}) := \mathrm H^{\bullet}(S_{\theta},j_{!*}\overline{\mathbb Q_{\ell}})$ for the intersection cohomology of $S_{\theta}$. It is a standard fact that $j_{!*}\overline{\mathbb Q_{\ell}}$ is pure of weight $0$, and therefore $\mathrm{IH}^k(S_{\theta})$ is pure of weight $k$ for all $k$. Since $Z$ is $0$-dimensional, it follows from \cite{perverse} Proposition 2.1.11 that the intersection complex is given by 
$$j_{!*}\overline{\mathbb Q_{\ell}} \simeq \tau_{\theta-1}\mathrm Rj_*\overline{\mathbb Q_{\ell}}.$$
The hypercohomology spectral sequence for intersection cohomology reads 
\begin{equation}\label{spectralbis}
F_2^{a,b} = \mathrm H^{a}(S_{\theta},\mathrm H^b(j_{!*}\overline{\mathbb Q_{\ell}})) \implies \mathrm{IH}^{a+b}(S_{\theta}).
\tag{$F$}
\end{equation}
For $k\geq 1$, the sheaf $\mathrm R^kj_{*}\overline{\mathbb Q_{\ell}}$ is skyscraper at the points of $Z$. Furthermore we have $j_*\overline{\mathbb Q_{\ell}} = \overline{\mathbb Q_{\ell}}$ since $S_{\theta}$ is irreducible and normal. It follows that 
$$
F_2^{a,b} = \begin{cases}
\mathrm H^{a}(S_{\theta}) & \text{if } b=0,\\
\bigoplus_{\overline{z}\in Z} (\mathrm R^bj_{*}\overline{\mathbb Q_{\ell}})_{\overline{z}} & \text{if } a = 0 \text{ and } 1 \leq b \leq \theta-1,\\
0 & \text{else.}
\end{cases}
$$
In particular, the spectral sequence degenerate in $F_{\theta+1}$, and we have $\mathrm H^k(S_{\theta}) = \mathrm{IH}^k(S_{\theta})$ for all $k > \theta$. The second page $F_2$ is drawn in Figure \ref{Figure2}. 

\begin{prop}\label{CohomologyUpperHalf}
For $k > \theta$, we have 
$$\mathrm H^k(S_{\theta}) \simeq 
\begin{cases}
A_0^{i,i}\oplus B_0^{i+1,i-1} & \text{if } k = 2i \text{ is even},\\
0 & \text{if } k \text{ is odd}.
\end{cases}$$
The statement also holds when $k = \theta$ is even.
\end{prop}

\begin{proof}
If $k > \theta$ is odd, the cohomology group $\mathrm H^k(S_{\theta}) = \mathrm{IH}^k(S_{\theta})$ is pure of weight $k$. Since the cohomology of $S_{\theta}$ consists of only even weights by Proposition \ref{EvenWeights}, we have $\mathrm H^k(S_{\theta}) = 0$. Assume now that $k = 2i \geq \theta$ is even. Given the repartition of the Frobenius weights in the first page of the spectral sequence \eqref{spectral}, we have $H^{2i}(S_{\theta}) \simeq \mathrm{Ker}(d_A^{i,i}) \oplus \mathrm{Ker}(d_B^{i+1,i-1})$, the first (resp. second) summand corresponding to the eigenvalue $q^i$ (resp. $-q^i$) of the Frobenius. Furthermore, since $H^{2i+1}(S_{\theta}) = 0$, the restriction of $d_A^{i,i}$ to $A_1^{i,i}$ (resp. of $d_B^{i+1,i-1}$ to $B_1^{i+1,i-1}$) defines an isomorphism onto $A_0^{i+1,i}$ (resp. onto $B_0^{i+2,i-1}$). Thus the kernel of $d_A^{i,i}$ (resp. of $d_B^{i+1,i-1}$) is reduced to $A_0^{i,i}$ (resp. to $B_0^{i+1,i-1}$).
\end{proof}

It remains to compute the cohomology of $S_{\theta}$ up to the middle degree. To do so, for $1 \leq k \leq \theta-1$ we consider the differential $\delta_k:F_{k+1}^{0,k} \to F_{k+1}^{k+1,0}$ in the $(k+1)$-th page of the spectral sequence \eqref{spectralbis}. We note that $F_{k+1}^{0,k} = F_2^{0,k}$ and $F_{k+1}^{k+1,0} = F_2^{k+1,0}$ since, up to the $(k+1)$-th page, both terms have not been touched by any non zero differential.

\begin{prop}\label{DifferentialsVanish}
The differential $\delta_k$ vanishes for odd $k$ and is surjective for even $k$.
\end{prop}

\begin{proof}
Assume that $k$ is odd. We have $F_2^{0,k} = \bigoplus_{\overline z \in Z} \mathrm H^{k}(j_{!*}\overline{\mathbb Q_{\ell}})_{\overline z}$. Since $j_{!*}\overline{\mathbb Q_{\ell}}$ is pure of weight $0$, the cohomology sheaf $\mathrm H^{k}(j_{!*}\overline{\mathbb Q_{\ell}})$ is mixed of weights $\leq k$. On the other hand, by Corollary \ref{Purity}, we know that $F_2^{0,k+1} = \mathrm H^{k+1}(S_{\theta})$ is pure of weight $k+1$. Therefore $\delta_k$ must vanish.\\
Assume now that $k$ is even. We know that $H^{k+1}(S_{\theta})$ is pure of even weight, whereas $\mathrm{IH}^{k+1}(S_{\theta})$ is pure of odd weight. Thus $\delta_k$ must be surjective. 
\end{proof}

\begin{proof}[Proof of Theorem \ref{MainTheorem}]
Let $k = 2i < \theta$ be even. Since the differential $\delta_{2i-1}$ vanishes, the term of coordinate $(2i,0)$ in the spectral sequence \eqref{spectralbis} is unchanged through the deeper pages. In particular, $\mathrm{IH}^{2i}(S_{\theta})$ contains a subspace isomorphic to $\mathrm H^{2i}(S_{\theta})$. Thus, we have  
$$\mathrm H^{2i}(S_{\theta}) \hookrightarrow \mathrm{IH}^{2i}(S_{\theta}) \simeq \mathrm{IH}^{2(\theta-i)}(S_{\theta})(\theta-2i),$$
where the isomorphism follows from the hard Lefschetz theorem for intersection cohomology. Since $2(\theta-i) > \theta$, the RHS is isomorphic to $A_0^{\theta-i,\theta-i}\oplus B_0^{\theta-i+1,\theta-i-1}$ by Proposition \ref{CohomologyUpperHalf}. But $A_0^{\theta-i,\theta-i} \simeq A_0^{i,i}$ and $B_0^{\theta-i+1,\theta-i-1} \simeq B_0^{i+1,i-1}$ as $\mathrm{Sp}(2\theta,\mathbb F_q)$-modules. As $\mathrm H^{2i}(S_{\theta})$ already contains $A_0^{i,i} \oplus B_0^{i+1,i-1}$, we actually have isomorphisms $\mathrm{IH}^{2i}(S_{\theta}) = \mathrm H^{2i}(S_{\theta}) \simeq A_0^{i,i} \oplus B_0^{i+1,i-1}$.\\
At this stage, we have proved statement (2) of Theorem \ref{MainTheorem}. According to Lemma \ref{Equivalence}, the proof is over.
\end{proof}

As a by-product, we have proved the following statement.

\begin{prop}
We have $j_{!*}\overline{\mathbb Q_{\ell}} \simeq \overline{\mathbb Q_{\ell}}$ and $\mathrm{IH}^k(S_{\theta}) = \mathrm H^k(S_{\theta})$ for all $k$.
\end{prop}

\begin{proof}
The natural map $\overline{\mathbb Q_{\ell}} \to j_{!*}\overline{\mathbb Q_{\ell}} \simeq \tau_{\theta-1}\mathrm Rj_*\overline{\mathbb Q_{\ell}}$ is a quasi-isomorphism, as we have incidentally proved that $\mathrm H^k(j_{!*}\overline{\mathbb Q_{\ell}}) = 0$ for all $k\geq 1$. 
\end{proof}

\phantomsection
\printbibliography[heading=bibintoc, title={Bibliography}]
\markboth{Bibliography}{Bibliography}

\appendix
\newgeometry{a4paper,left=1in,right=1in,top=1in,bottom=1in,nohead}
\begin{landscape}
\thispagestyle{empty}
\section{Figures}
\vfill
\begin{figure}[h]
\centering
\begin{tikzcd}[sep=small]
	\, & \, & \, & \, & \, & \, & A_0^{\theta,\theta}\\
	\, & \, & \, & \, & \, &  A_0^{\theta-1,\theta-1} \oplus A_1^{\theta-1,\theta-1}  \arrow{r} &  A_0^{\theta,\theta-1}\\
	\, & \, & \, & \, &  A_0^{\theta-2,\theta-2} \oplus A_1^{\theta-2,\theta-2} \arrow{r} &  A_0^{\theta-1,\theta-2} \oplus A_1^{\theta-1,\theta-2}  \arrow{r} &  A_0^{\theta,\theta-2} \oplus B_0^{\theta,\theta-2}\\
	\, & \, & \, & \reflectbox{$\ddots$} & \, & \, & \vdots\\
    \, & \, &  A_0^{2,2} \oplus A_1^{2,2}  \arrow{r} & \ldots \arrow{r} & \begin{array}{c} A_0^{\theta-2,2} \oplus A_1^{\theta-2,2} \\ B_0^{\theta-2,2} \oplus B_1^{\theta-2,2} \end{array} \arrow{r} & \begin{array}{c} A_0^{\theta-1,2} \oplus A_1^{\theta-1,2} \\ B_0^{\theta-1,2} \oplus B_1^{\theta-1,2} \end{array} \arrow{r} & A_0^{\theta,2} \oplus B_0^{\theta,2} \\
    \, &  A_0^{1,1} \oplus A_1^{1,1}  \arrow{r} & A_0^{2,1} \oplus A_1^{2,1} \arrow{r} & \ldots \arrow{r} &\begin{array}{c}  A_0^{\theta-2,1} \oplus A_1^{\theta-2,1} \\ B_0^{\theta-2,1} \oplus B_1^{\theta-2,1} \end{array} \arrow{r} & \begin{array}{c} A_0^{\theta-1,1} \oplus A_1^{\theta-1,1} \\ B_0^{\theta-1,1} \oplus B_1^{\theta-1,1} \end{array} \arrow{r} & A_0^{\theta,1} \oplus B_0^{\theta,1} \\
    A_0^{0,0} \oplus A_1^{0,0} \arrow{r} & A_0^{1,0} \oplus A_1^{1,0} \arrow{r} & \begin{array}{c} A_0^{2,0} \oplus A_1^{2,0} \\ B_0^{2,0} \oplus B_1^{2,0} \end{array} \arrow{r} & \ldots \arrow{r} & \begin{array}{c} A_0^{\theta-2,0} \oplus A_1^{\theta-2,0} \\ B_0^{\theta-2,0} \oplus B_1^{\theta-2,0} \end{array} \arrow{r} & \begin{array}{c} A_0^{\theta-1,0} \oplus A_1^{\theta-1,0} \\ B_0^{\theta-1,0} \oplus B_1^{\theta-1,0} \end{array} \arrow{r} & A_0^{\theta,0} \oplus B_0^{\theta,0}
\end{tikzcd}
\caption{The first page of the spectral sequence \eqref{spectral}.}\label{Figure1}
\end{figure}
\vfill

\newpage

\thispagestyle{empty}
\vfill
\begin{figure}
\centering
\begin{tikzcd}
\bigoplus_{\overline{z}\in Z} (\mathrm R^{\theta-1}j_{*}\overline{\mathbb Q_{\ell}})_{\overline{z}} \arrow[dashed]{rrrrddd} & & & & & & &\\
\vdots \arrow[dashed]{rrrdd} & & & & & & &\\
\bigoplus_{\overline{z}\in Z} (\mathrm R^1j_{*}\overline{\mathbb Q_{\ell}})_{\overline{z}} \arrow{rrd} & & & & & & &\\
\mathrm H^0(S_{\theta}) & \mathrm H^1(S_{\theta}) & \mathrm H^2(S_{\theta}) & \ldots & \mathrm H^{\theta}(S_{\theta}) & \mathrm H^{\theta+1}(S_{\theta}) & \ldots & \mathrm H^{2\theta}(S_{\theta})
\end{tikzcd}
\caption{The second page of the spectral sequence \eqref{spectralbis} (the differentials in dashed lines correspond to deeper pages).}\label{Figure2}
\end{figure}
\vfill
\end{landscape}

\end{document}